\documentclass[reqno,A4paper]{amsart}
\usepackage{amsmath,amssymb,amscd,amsthm,verbatim,alltt,amsfonts,array}
\usepackage{mathrsfs}
\usepackage[english]{babel}
\usepackage{latexsym}
\usepackage{amssymb}
\usepackage{euscript}
\usepackage{graphicx}
\usepackage{csquotes}
\newtheorem{thm}{Theorem}[section]
\newtheorem{cor}[thm]{Corollary}
\newtheorem{prop}[thm]{Proposition}

\newtheorem{lem}[thm]{Lemma}
\newtheorem{Def}[thm]{Definition}
\newtheorem{rem}[thm]{Remark}

\newcommand{\be}{\begin{equation}}
\newcommand{\ee}{\end{equation}}
\newcommand{\ben}{\begin{enumerate}}
\newcommand{\een}{\end{enumerate}}
\newcommand{\beq}{\begin{eqnarray}}
\newcommand{\eeq}{\end{eqnarray}}
\newcommand{\beqn}{\begin{eqnarray*}}
\newcommand{\eeqn}{\end{eqnarray*}}

\title{On the Generalized Projective Riemann Curvature in Finsler Geometry}
\author{Nasrin Sadeghzadeh$^*$}
\newcommand{\acr}{\newline\indent}
\address{\llap{*\,}Department of Mathematics, University of Qom,\acr
Alghadir Bld, Qom, Iran}
\email{nsadeghzadeh@qom.ac.ir}
\author{Masoumeh Yaghoubi}
\address{Department of Mathematics, University of Qom,\acr
Alghadir Bld, Qom, Iran}
\email{massoumeh.yaghoubi@stu.qom.ac.ir}

\begin{document}
\maketitle
\begin{center}
\textbf{This is a preprint of an article currently under review.}
\end{center}

\section{Abstract}
This paper explores the generalized projective Riemann curvature in Finsler geometry, focusing on the properties of projectively equivalent Finsler metrics and the invariance of their curvature structures under projective transformations. We extend the existing frameworks of projective Riemann and Ricci curvatures by introducing new characterizations of quadratic curvature properties, highlighting their geometric significance in the broader context of Finsler manifolds. Our results provide novel insights into curvature behavior under generalized projective sprays and contribute to a deeper understanding of intrinsic geometric invariants within projective classes of Finsler spaces.
\\
\subjclass{MSC2020:} {53B40; 53C60}
\\
\keywords{Projective change, $GDW$-metrics, $R$-quadratic Finsler metrics, $GPR$-quadratic Finsler metrics.}

\section{Introduction}

The concept of projective changes between two Finsler spaces describes transformations that preserve the geodesic curves as point sets. This important notion has been extensively studied by many researchers \cite{BacsoMat-Proj, CuBingShen-Proj, ParkLee-Proj, Rapcsak-Proj}, revealing deep insights into the geometry and invariant quantities of Finsler spaces. Projective invariance plays a fundamental role in understanding equivalence classes of Finsler metrics and the intrinsic geometry of paths, which is central to the theory of Finsler structures \cite{Sh2}. These invariants are essential for applications in geometric analysis and mathematical physics, where the geometry of curves underlies phenomena such as geometric optics and variational problems \cite{ShenPRicFlat}. Notably, tensorial and scalar projective invariants were explored in \cite{Sinha, BidabadSepasi}, while \cite{SaMaRa} investigated projective invariants that remain stable under conformal changes of metrics.

Two Finsler metrics $F$ and $\bar{F}$ on a manifold $M$ are said to be projectively related if their geodesic coefficients satisfy the relation $G^i = \bar{G}^i + P y^i$, where $G^i$ and $\bar{G}^i$ are the geodesic coefficients and $P$ is the projective factor. The study of such metrics forms the core of projective Finsler geometry, exploring equivalence classes of metrics sharing the same geodesics as point sets and providing insights into the geometric structures invariant under re-parameterizations of geodesics \cite{ParkLee-Proj}. Unlike several curvature quantities such as the Douglas, Weyl, and Generalized Douglas-Weyl ($GDW$) curvatures, as well as some recently introduced projectively invariant curvatures presented in \cite{SadeghNewBirth}, many important curvature properties in Finsler geometry do not remain invariant under projective changes. Of particular interest is the class of $R$-quadratic Finsler metrics, which are not closed under projective changes, although they lie within the projectively closed class of Generalized Douglas-Weyl metrics. For example, as shown in \cite{NajBidTay}, \cite{Sh2}, the Randers metric
\[
F=|y|+\frac{\langle x,y \rangle}{\sqrt{1+|x|^2}},
\]
where $| \cdot |$ and $\langle \cdot,\cdot \rangle$ denote the Euclidean norm and inner product on $\mathbb{R}^n$, is of scalar flag curvature but not $R$-quadratic, despite being projectively related to an $R$-quadratic Euclidean metric.

Classically, the projective spray and the induced projective Riemann and Ricci curvatures in Finsler geometry are obtained by deforming the spray \(G\) with the projective factor \(P = \frac{S}{n+1}\), where \(S\) denotes the \(S\)-curvature defined relative to a fixed volume form on the manifold \cite{Sh2, ShenPRicFlat}. This choice guarantees that the constructed curvature tensors remain invariant under projective changes that preserve both the geodesic paths and the volume structure. The \(S\)-curvature plays a significant role beyond volume distortion, affecting the geometric flows and stability of curvature conditions in Finsler manifolds \cite{ShenPRicFlat, TabNajTay}.

A particularly important object in this context is the projective Ricci curvature, introduced by Z. Shen and collaborators , which refines the classical Ricci curvature by accounting for projective changes and volume forms. This curvature has remarkable invariance properties under projective transformations of the geodesic spray associated with a Finsler metric. Shen \cite{Sh2} introduced the projective Ricci curvature by considering the Ricci and $S$-curvatures associated with a Finsler metric $F$. It has been shown that if two Finsler metrics are point-wise projectively related on a manifold with a fixed volume form, then their projective Ricci curvatures coincide. In other words, projective Ricci curvature is invariant under projective transformations respecting a fixed volume form on the manifold.

This curvature can be regarded as a weighted Ricci curvature, a concept further extended in \cite{TabNajTay} through the notion of weighted projective Ricci curvature, offering refined tools to analyze the geometry of Finsler spaces.

Extensive studies have utilized projective Ricci curvature to characterize projectively Ricci-flat spaces, isotropic projective Ricci curvature, and properties based on projective Riemann curvature, significantly enriching the classification of Finsler manifolds \cite{ShenPRicFlat, IsoPRic, SadeghPR}. Recent advances, including weighted generalizations and novel curvature conditions \cite{TabNajTay}, \cite{ZhuPR}, \cite{ShenPRicFlat}, show that this classical framework can be extended by allowing the projective factor to be a more general homogeneous function \(\rho_G = \rho_G(x,y)\). This generalization captures richer geometric phenomena by incorporating anisotropies and alternative volumetric or curvature-related influences beyond the classical volume distortion encoded by \(S\).
This leads to generalized projective sprays whose geodesic re-parameterizations reflect more intricate curvature-driven behaviors, enabling refined projective invariants that better encode the complexity of Finsler geometries.

In fact, the Riemann and Ricci curvatures are fundamental geometric invariants in Riemannian geometry quantifying intrinsic curvature and volume distortion of manifolds. The Riemann curvature tensor measures directional changes of vectors after parallel transport along curves, encoding sectional curvature information. The Ricci curvature, obtained as a trace of the Riemann tensor, represents the average deformation of volume elements along geodesic flows, playing a central role in theoretical physics, especially in Einstein’s theory of general relativity. There, the Ricci curvature directly relates to the distribution of matter and energy via Einstein’s field equations. Additionally, sprays represent geometric objects that encapsulate geodesic equations as systems of second-order differential equations, serving as fundamental tools in geometric analysis and dynamical systems modeling physical trajectories.

Extending these notions to Finsler geometry leads to a more general and direction-dependent framework where metric functions vary with both position and direction. This anisotropy allows for modeling spaces with direction-dependent characteristics, reflecting more realistic physical situations such as anisotropic media and generalized space-time structures. In Finslerian settings, sprays and curvature tensors are naturally generalized to capture this directional dependence and inherent complexity.

Classically, projective sprays $G$, together with their associated projective Riemann and Ricci curvatures, have been defined by deforming spray structures using the $S$-curvature—an important volume distortion measure. These projective curvatures are invariant under projective changes and provide refined geometric invariants that are crucial for understanding volume-preserving geometric flows and anisotropic or generalized gravitational models.
This work further generalizes these constructions by replacing the volume-related $S$-curvature with a broader class of smooth, positively 1-homogeneous functions \(\rho_G \) dependent on the spray coefficients but not restricted to volume distortion factors. This generalization vastly expands the class of admissible projective sprays and their curvature invariants, introducing new geometric and physical possibilities. Physically, such a general \(\rho_G \) allows modeling anisotropic or inhomogeneous effects in media or space-time where directional dependencies or volume distortions are more complex than classical scenarios.
Consequently, the generalized projective Ricci curvature defined through \(\rho_G \) encodes refined curvature phenomena. This curvature can describe anisotropic transport properties, directional gravitational effects, and complex optical behavior in inhomogeneous media.
In a Finslerian generalization of general relativity, this generalized projective Ricci curvature could be related to the stress-energy tensor of anisotropic matter fields, providing a more refined geometric description of gravity compared to the standard Ricci tensor.
The model appears to have potential applications in cosmology (e.g., anisotropic dark matter models) or condensed matter physics.
This framework offers a powerful geometric tool set to explore and model intricate physical environments where classical isotropic geometries prove insufficient, paving the way for advancements in the mathematical modeling of direction-dependent physical theories and the extension of geometric analysis.

Consequently, the generalized projective Riemann and Ricci curvatures associated with these sprays provide a natural and powerful extension of classical projective invariants with broader applicability. They enable a deeper understanding of Finsler metrics exhibiting complex projective symmetries, anisotropic curvatures, and variable volume growth, significantly advancing the frontier of projective Finsler geometry.

In this paper, we propose a comprehensive generalized framework for the concept of projective Riemann curvature in Finsler geometry, aiming to unify and extend existing theories on projective Ricci and projective Riemann curvature in Finsler metrics. Building upon the foundational work of Z. Shen and others on projective Ricci curvature and its invariance properties , we develop new characterizations and structural results for generalized projective Riemann and Ricci quadratic metrics. Our approach introduces a novel class of generalized projective geodesic coefficients which preserve the projective Riemann curvature tensor under projective transformations, thereby deepening the understanding of geometrical invariants in the projective class of Finsler spaces.
Furthermore, we present several illustrative examples that effectively showcase the breadth and applicability of this generalization.

\noindent The main results include the following theorems.

\begin{thm}\label{Thm1}
A Finsler metric $F$ on $n$-dimensional manifold $M$ is generalized projectively $R$-quadratic ($GPR$-quadratic) with respect to $\rho_G$ if and only if
\be \label{GPR-quadEq}
PB_j{}^{i}{}_{kl|0} - \rho_{G\,.r} (PB_j{}^{r}{}_{kl}) y^{i}=0,
\ee
where $PB_j{}^{i}{}_{kl}$ is the Berwald curvature of projective spray $PG$ with respect to $\rho_G$.
\end{thm}

A direct consequence of the above theorem is the following corollary.

\begin{cor}\label{PR-quadCor} \cite{SadeRoleS-cur}
A Finsler metric on $n$-dimensional manifold $M$ is $PR$-quadratic if and only if
\[
D_j{^i}_{kl|0}=\frac{S_{.r}}{n+1}D_j{^r}_{kl} y^i.
\]
\end{cor}

In this paper, we utilize symbols such as \enquote{$._{}$} and \enquote{$|_{}$} to indicate vertical and horizontal derivatives concerning the Berwald connection. Furthermore, the subscript \enquote{${}_0$} signifies the contraction with respect to $y^m$, as indicated by the subscript \enquote{${}_{m}$}, while the notation \enquote{${}_{;m}$} represents the derivative taken with respect to $x^m$.

\section{Preliminaries}
A Finsler manifold is a smooth manifold \(M\) endowed with a Finsler metric \(F: TM \to [0,\infty)\), which is smooth on the slit tangent bundle \(TM \setminus \{0\}\). The fundamental tensor
\[
{\bf g}_y(u,v) := \frac{1}{2} \frac{\partial^2}{\partial s\, \partial t} \Big|_{s,t=0} F^2(y + s u + t v)
\]
is positive definite for all \(y \neq 0\) in \(T_x M\) and all \(u,v \in T_x M\) \cite{Sh2}. This tensor generalizes the Riemannian metric tensor by allowing direction dependence. If for each \(x\), \({\bf g}_y\) is independent of \(y\), then \(F\) reduces to a Riemannian metric at \(x\).

The geodesics of \((M,F)\) are the critical points of the energy functional associated to \(F\) and are locally minimizing curves up to re-parametrization \cite{Sh2}. The functions \(G^i\) define the \textit{geodesic spray} vector field
\[
G = y^i \frac{\partial}{\partial x^i} - 2 G^i \frac{\partial}{\partial y^i},
\]
whose integral curves project to geodesics on \(M\) \cite{Sh2}.

The Riemann curvature tensor \(R^i_k\) of the geodesic spray encodes the curvature of the Finsler structure.
\[
R^i_k = 2 \frac{\partial G^i}{\partial x^k} - \frac{\partial^2 G^i}{\partial x^m \partial y^k} y^m + 2 G^m \frac{\partial^2 G^i}{\partial y^m \partial y^k} - \frac{\partial G^i}{\partial y^m} \frac{\partial G^m}{\partial y^k},
\]
and satisfies antisymmetry and Bianchi identities \cite{Sh2}.
\[
R^i_{kl} = \frac{1}{3} (R^i_{k.l} - R^i_{l.k}), \quad R_j{}^i_{kl} = R^i_{kl.j}.
\]
A Finsler metric is called Berwaldian if the spray coefficients \(G^i\) are quadratic in the fiber variables \cite{Sh2}.

The Berwald curvature tensor is given by
\[
B_y(u,v,w) = B_j{}^i_{kl} u^j v^k w^l \frac{\partial}{\partial x^i},
\]
with components
\[
B_j{}^i_{kl} = \frac{\partial^3 G^i}{\partial y^j \partial y^k \partial y^l}.
\]
It satisfies the fundamental identity relating it to the Riemann curvature:
\be\label{Rie-BerId}
B_j{}^i_{ml|k} - B_j{}^i_{mk|l} = R_j{}^i_{kl.m},
\ee
where \(\ |\) and \(\ .m\) denote horizontal and vertical covariant derivatives with respect to the Berwald connection, respectively \cite{Sh2}.

The mean Berwald curvature, or \(E\)-curvature, is defined as
\[
E_{jk} := \frac{1}{2} B_j{}^m_{km},
\]
and Finsler metrics with zero Berwald curvature (\(B=0\)) are Berwald metrics, while those with zero mean Berwald curvature (\(E=0\)) are called weakly Berwald \cite{Sh3}. Isotropic mean Berwald curvature means
\[
E_{ij} = \frac{n+1}{2} c F^{-1} h_{ij},
\]
for some scalar function \(c=c(x)\), where \(h_{ij}\) is the angular metric.
The \(S\)-curvature measures the rate of change of the volume distortion along geodesics:
\[
S(x,y) := \left. \frac{d}{dt} \tau(\gamma(t), \dot{\gamma}(t)) \right|_{t=0},
\]
where \(\tau(x,y)\) is the distortion function for \(F\) and \(\gamma(t)\) is the geodesic with \(\gamma(0)=x\), \(\dot{\gamma}(0) = y\) \cite{Sh3}. A key formula relates the mean Berwald curvature and $S$-curvature by
\[
E_{ij} = \frac{1}{2} S_{.i.j},
\]
where the subscript \(".i"\) denotes differentiation with respect to \(y^i\) \cite{Sh2}.

The Douglas tensor \(D\) is defined by
\[
D_j{}^i_{kl} := B_j{}^i_{kl} - \frac{1}{n+1} \frac{\partial^3}{\partial y^j \partial y^k \partial y^l} \left( \frac{\partial G^m}{\partial y^m} y^i \right).
\]
This tensor is a fundamental projective invariant on the slit tangent bundle \(TM_0\) and is preserved under projective transformations of the spray \cite{Sh2}. Equivalently,
\[
D_j{}^i_{kl} = B_j{}^i_{kl} - \frac{2}{n+1} \left( E_{jk} \delta^i_l + E_{jl} \delta^i_k + E_{kl} \delta^i_j + E_{jk.l} y^i \right).
\]
Finsler metrics with vanishing Douglas tensor are called Douglas metrics. Those satisfying
\[
D_j{}^i_{kl|m} y^m = T_{jkl} y^i,
\]
for some tensor \(T_{jkl}\), are known as Generalized Douglas-Weyl (GDW) metrics, which are projectively invariant \cite{Sh2}.

\begin{lem} \cite{Sh2}
If two Finsler metrics \(F\) and \(\bar{F}\) are projectively related by \(G^i = \bar{G}^i + P y^i\), then their Riemann curvatures satisfy
\[
\bar{R}{}^i_k = R^i_k + E \delta^i_k + \tau_k y^i,
\]
where
\[
E = P^2 - P_{|m} y^m, \quad \tau_k = 3 (P_{|k} - P P_{.k}) + E_{.k},
\]
and \(P_{|k}\) is the horizontal covariant derivative of the projective factor \(P\) with respect to \(\bar{F}\).
\end{lem}

A projective transformation with factor \(P\) is called C-projective if the tensor
\[
Q_{ij} = \frac{\partial Q_i}{\partial y^j} - \frac{\partial Q_j}{\partial y^i},
\]
with
\[
Q_i = \frac{\partial P}{\partial x^i} - G^m_{.i} \frac{\partial P}{\partial y^m} - P \frac{\partial P}{\partial y^i},
\]
vanishes. Under this condition, when \(P=-\frac{S}{n+1}\), we obtain the classical projective Riemann curvature \cite{Sh2}.
For a spray $G$ on an $n$-dimensional manifold $M$ and given a volume form $dV$ on $M$, we can construct a new spray by
\[
\tilde{G}:=G+\frac{2S}{n+1}Y.
\]
The spray $\tilde{G}$ associated with a Finsler metric $(M,F)$ and its corresponding volume form $dV$ is referred to as the projective spray. This terminology holds in both local and global contexts.
\be\label{PSpraylocal}
\widetilde{G}^i=G^i-\frac{S}{n+1}y^i.
\ee
The projective Ricci curvature of $(G, dV)$ is defined as the Ricci curvature of the associated projective spray $\tilde{G}$, namely,
\[
PRic_{(G, dV)}:=Ric_{\tilde{G}}.
\]
Then by a simple computation one has
\be\label{PRic}
PRic_{(G,dV)}=Ric+(n-1)\{\frac{S_{|0}}{n+1}+[\frac{S}{n+1}]^2\}.
\ee
where $Ric = Ric_G$ is the Ricci curvature of the spray $G$, $S=S_{(G,dV)}$ is the S-curvature of $(G,dV)$ and $S_{|0}$ is the covariant derivative of $S$  along a geodesic of $G$. It is worth noting that $\tilde{G}$ remains unchanged under a projective change of $G$ with $dV$ fixed. As a result, $PRic_{(G,dV)}=Ric_{\tilde{G}}$ is a projective invariant of $(G,dV)$.\\
Given a Finsler metric $(M,F)$, the Riemann curvature of a projective spray, $PR_y=PR{^i}_k \frac{\partial }{\partial x^i} \otimes dx^k$, is referred to as the projective Riemann curvature,
\[
{PR{^i}_k}_{(G,dV)}= {R{^i}_k}_{\widetilde{G}}.
\]
A Finsler metric $(M,F)$ is called $PR$-quadratic Finsler metric if $PR_j{^i}_{kl.m}=0$.

Throughout this paper, the symbols ${}_{.}$ and ${}_{|}$ denote the vertical and horizontal covariant derivatives with respect to the Berwald connection, respectively, while the subscript ${}_0$ indicates contraction by \(y^m\), that is, \({}_{|m} y^m\).

\section{Generalized Projective Riemann and Ricci curvatures}

In this section, we introduce new definitions of generalized projective sprays, and the associated Riemann and Ricci curvatures, tailored for Finsler manifolds. Our aim is to extend the classical notions of projective invariants in Finsler geometry. Traditionally, the projective Riemann and Ricci curvatures are defined using the classical projective spray constructed by deforming the geodesic spray with the projective factor \( P = \frac{S}{n+1} \), where \( S \) denotes the \(S\)-curvature associated to a fixed volume form on the manifold. However, the classical role of \(S\)-curvature as volume distortion measure may be generalized by considering alternative smooth 1-homogeneous functions that capture different geometric or physical characteristics of the manifold. By replacing the \(S\)-curvature with other suitable homogeneous functions, we can construct generalized projective sprays and curvature notions, thereby broadening the class of geometric structures that can be modeled. Such generalizations enable the exploration of diverse curvature behaviors beyond volume distortion, offering new insights into projective equivalences and geometric flows in Finsler spaces.

We proceed by formally defining the generalized projective Riemann curvature and the corresponding generalized projective Ricci curvature. These definitions allow us to capture finer geometric invariants that are stable under an extended family of projective transformations.
As is known \cite{Sh2, ShenPRicFlat}, a projective spray \(G\) on an \(n\)-dimensional manifold \(M\) remains invariant under projective changes. We denote such a spray, fixed under projective transformations, by \(PG\), and refer to it as the projective spray.

To generalize, let \(G\) be a spray on an \(n\)-dimensional manifold \(M\). We define a new spray by
\[
PG = G + 2 \rho_G Y,
\]
where \(Y = y^i \frac{\partial}{\partial y^i}\) is the vertical radial vector field on \(TM\), and \(\rho_G = \rho_G(x,y)\) is a smooth, 1-homogeneous function associated with \(G\). We seek the conditions on \(\rho_G\) such that \(PG\) remains invariant under all projective changes of \(G\).

In local coordinates, if \(\tilde{G}^i\) are the geodesic coefficients of a spray projectively related to \(G^i\) by the projective factor \(P=P(x,y)\), i.e.,
\[
\tilde{G}^i = G^i - P y^i,
\]
then it must hold that
\[
P \tilde{G}^i = PG^i,
\]
where by definition
\[
P \tilde{G}^i := \tilde{G}^i - \rho_{\tilde{G}} y^i.
\]
This equality is equivalent to the consistency condition
\[
\rho_{\tilde{G}} = \rho_G + P.
\]
We call \( PG \) a generalized projective spray associated with \( \rho_G \).

\begin{Def}\label{GPSprayDef}
Let \(G\) be a spray on an \(n\)-dimensional Finsler manifold \(M\). Suppose \(\rho_G = \rho_G(x,y)\) is a smooth, degree 1 homogeneous function. Then
\[
PG^i = G^i - \rho_G y^i
\]
defines a generalized projective spray with respect to \(\rho_G\) if and only if for every projective change
\[
\tilde{G}^i = G^i + P y^i,
\]
the compatibility condition
\[
\rho_G - \rho_{\tilde{G}} = P
\]
holds.
\end{Def}

There can be multiple generalized projective sprays associated with the same spray \(G\). For instance, the following two classes provide important examples.

\begin{prop}\label{rhoPropGS-cur}
Let \((M,F)\) be an \(n\)-dimensional Finsler manifold equipped with a volume form \(dV_F = \sigma(x) dx^1 \wedge \cdots \wedge dx^n\), and a fixed volume form \(dV_0 = \sigma_0(x) dx^1 \wedge \cdots \wedge dx^n\) associated with a reference Finsler metric \(F_0\). For any 1-homogeneous function \(U=U(x,y)\), the spray
\[
PG^i = G^i - \left( \frac{\mathbb{S}}{n+1} + U \right) y^i,
\]
defines a generalized projective spray, where \(\mathbb{S} = S + d \ln \left( \frac{\sigma_0}{\sigma} \right)\), with \(S\) the \(S\)-curvature corresponding to \(F\).
\end{prop}

\begin{proof}

This spray \(PG\) remains invariant under projective changes of the form
\[
\tilde{G}^i = G^i + P y^i,
\]
with the volume form \(dV_F = \sigma(x) dx^1 \wedge \cdots \wedge dx^n\), and a fixed volume form \(dV_0 = \sigma_0(x) dx^1 \wedge \cdots \wedge dx^n\).
Noting that
\[
\tilde{S} = \tilde{G}^m_{.m} - d \ln(\tilde{\sigma}) = G^m_{.m} + (n+1) P - d \ln(\tilde{\sigma}) = S + (n+1) P + d \ln\left( \frac{\sigma_0}{\sigma} \right) - d \ln\left( \frac{\sigma_0}{\tilde{\sigma}} \right).
\]
Therefore,
\[
\frac{1}{n+1}(\tilde{\mathbb{S}} - \mathbb{S}) = P.
\]
More precisely, for any 1-homogeneous projective invariant function \(U = U(x,y)\), we obtain
\[
\rho_{\tilde{G}} = \frac{\tilde{\mathbb{S}}}{n+1} + U = \frac{\mathbb{S}}{n+1} + U + P = \rho_G + P,
\]
which confirms the condition in Definition \ref{GPSprayDef}.
\end{proof}

The above proposition naturally includes the weighted projective Ricci curvature framework introduced by \cite{TabNajTay} as a special case (with \(U=0\)). Accordingly, the following corollary holds.
\begin{cor}\label{GPSprayGS-cur}
Let \((M,F)\) be as above. For any smooth function \(f=f(x)\) on \(M\), the spray
\[
PG^i = G^i - \left( \frac{\mathbb{S}}{n+1} + f_{;0} \right) y^i,
\]
defines a generalized projective spray, where the notation is as before.
\end{cor}

Replacing \(S\)-curvature by suitable homogeneous functions thus broadens the construction of projective sprays and curvature concepts, enlarging the geometric horizon beyond volume distortion. Alternative 1-homogeneous functions represent varied geometric or physical properties, thus enabling new curvature invariants. We now introduce another projective spray and analyze its properties, starting with the following definition.

\begin{Def}\label{DPIDef}
Let \(V = V(x,y)\) be a 1-homogeneous function on the Finsler manifold \((M,F)\).
We define the spray directional derivative of \(V\) by
\[
\mathcal{D} V := (\ln V)_{.r} G^r,
\]
where \(G^r\) are the spray coefficients associated with the Finsler metric \(F\).
The function \(V\) is called directionally projectively invariant if for any projective change of spray \(\tilde{G}^i = G^i + P y^i\), with corresponding transformed function \(\tilde{V}\), we have
\[
\mathcal{D}_{\tilde{V}} V = \left(\ln \frac{\tilde{V}}{V} \right)_{.r} G^r = 0.
\]
\end{Def}

This condition signifies that \(V\) remains invariant under infinitesimal displacements along the geodesics defined by the spray under projective transformations, reflecting an intrinsic directional invariance with respect to the underlying projective structure.

\begin{rem}
Every 1-homogeneous function \(V=V(x,y)\) which is projectively invariant is also directionally projectively invariant.
\end{rem}

\begin{prop}\label{rhoPropDV}
Let \((M,F)\) be an \(n\)-dimensional Finsler manifold with induced spray \(G\). For any 1-homogeneous, directionally projectively invariant function \(V=V(x,y)\), the deformation
\[
PG^i := G^i - \mathcal{D}_V y^i,
\]
where \(\mathcal{D}_V\) is the spray directional derivative of \(V\), defines a generalized projective spray.
\end{prop}

\begin{proof}
Since \(V\) is directionally projectively invariant, under any projective change \(\tilde{G}^i = G^i + P y^i\), the spray directional derivatives satisfy
\[
\tilde{V}_{.r} \tilde{G}^r = \tilde{V}_{.r} G^r + P \tilde{V}.
\]
Hence,
\[
\rho_{\tilde{G}} - \rho_G = (\ln \tilde{V})_{.r} \tilde{G}^r - (\ln V)_{.r} G^r = \left[ \ln \tilde{V} - \ln V \right]_{.r} G^r + P = \mathcal{D}_{\tilde{V}} V + P.
\]
By the directional projective invariance of \(V\), \(\mathcal{D}_{\tilde{V}} V=0\), satisfying the condition in Definition \ref{GPSprayDef}.
\end{proof}

Based on these generalized sprays, the concepts of generalized projective Riemann and Ricci curvatures can be extended \cite{SadeghPR, ShenPRicFlat, ZhuPR}. In particular, we define

\begin{Def}\label{GPRicSprayDef}
Let \(G\) be a spray on an \(n\)-dimensional Finsler manifold \(M\). The generalized projective Riemann curvature with respect to \(\rho_G\) is defined as the Riemann curvature \(\{PR{^i}_k\}\) of the generalized projective spray \(PG\).
\end{Def}
\begin{Def}\label{GPR-quadSprayDef}
Let \(G\) be a spray on an \(n\)-dimensional Finsler manifold \(M\). The generalized projective Ricci curvature with respect to \(\rho_G\) is defined as the Ricci curvature
\[
PRic = Ric + \rho_G^2 + \rho_{G|\;k}.
\]
\end{Def}

\begin{Def}\label{GPR-quadSprayDef}
Let \( G \) be a spray on an \( n \)-dimensional Finsler manifold \( M \). Generalized projective Ricci curvature with respect to $\rho_G$ is defined as the Ricci curvature \(PRic\) of the associated generalized projective spray with respect to \(\rho_G\). In particular,
\be\label{PRicDef}
PRic=Ric + \rho_G^2+\rho_G\,_{\mid k}.
\ee
\end{Def}

With these definitions and the propositions \ref{rhoPropGS-cur} and \ref{rhoPropDV}, the generalized projective Riemann and Ricci curvatures take the following explicit forms.
\begin{Def}\label{PRicDefGS-Cur}
Let \(G\) be a spray on an \(n\)-dimensional Finsler manifold \(M\). For any 1-homogeneous projective invariant function \(U = U(x,y)\), the generalized projective Ricci curvature is
\[
PRic = Ric + \frac{1}{n+1} \left( \frac{\mathbb{S}^2}{n+1} + \mathbb{S}_{|0} \right) + U^{2} + U_{|0},
\]
where \(\mathbb{S}\) is as introduced in Proposition \ref{rhoPropGS-cur}.
\end{Def}
\begin{Def}\label{PRicDefDV}
Let \(G\) be a spray on an \(n\)-dimensional Finsler manifold \(M\). For any 1-homogeneous, directionally projectively invariant function \(V = V(x,y)\), the generalized projective Ricci curvature is
\[
PRic = Ric + \frac{1}{n+1} \left( (\mathcal{D}V)^2 + \mathcal{D}V_{|0} \right),
\]
where \(\mathcal{D}V\) is as introduced in Proposition \ref{rhoPropDV}.
\end{Def}

Finally, we define important classes of generalized projective Finsler metrics as follows:
\begin{Def}
A Finsler metric \(F\) on \(M\) is called generalized projectively Ricci quadratic (respectively, Ricci-flat) with respect to \(\rho_G\) if the induced spray \(PG\) associated to \(\rho_G\) is projectively Ricci quadratic (respectively, Ricci-flat).
\end{Def}
\begin{Def}
A Finsler metric \(F\) is generalized projectively \(R\)-quadratic with respect to \(\rho_G\) if the induced spray \(PG\) corresponding to \(\rho_G\) is projectively \(R\)-quadratic.
\end{Def}
\begin{rem}\label{Remark1}
For brevity, we denote these notions by \(GPRic\)-quadratic and \(GPR\)-quadratic respectively.
\end{rem}

\subsection{Proof of Theorem \ref{Thm1}}

\begin{proof}
Assume that \(G\) is a spray on an \(n\)-dimensional Finsler manifold \((M,F)\). Consider the generalized projective spray \(PG\) defined by
\be\label{Gi}
PG^i = G^i - \rho_G y^i,
\ee
where \(\rho_G = \rho_G(x,y)\) is a smooth, 1-homogeneous function on \(TM\). Differentiating successively with respect to \(y^j\), \(y^k\), and \(y^l\), we obtain
\begin{equation}\label{Gij}
PG^{i}_{j} = G^{i}_{j} - \rho_G \delta^{i}_{j} - \rho_{G\,.j} y^{i},
\end{equation}
\begin{equation}\label{Gijk}
PG^{i}_{jk} = G^{i}_{jk} - \rho_{G\,.k} \delta^{i}_{j} - \rho_{G\,.j} \delta^{i}_{k} - \rho_{G\,.j.k} y^{i},
\end{equation}
and
\begin{equation}\label{Bjikl}
PB_j{}^{i}{}_{kl} = B_j{}^{i}{}_{kl} - \rho_{G\,.k.l} \delta^{i}_{j} - \rho_{G\,.j.l} \delta^{i}_{k} - \rho_{G\,.k.j} \delta^{i}_{l} - \rho_{G\,.j.k.l} y^{i}.
\end{equation}

Recall the Berwald curvature of the spray \(PG\) is defined by
\[
PB_j{}^i{}_{kl} = \frac{\partial^3 PG^i}{\partial y^j \partial y^k \partial y^l}.
\]
Using the above differentiation, we have
\begin{equation}\label{horizontalBandBbarGbar}
PB_j{}^{i}{}_{kl||m} = B_j{}^{i}{}_{kl||m} - \rho_{G\,.k.l||m} \delta^{i}_{j} - \rho_{G\,.j.l||m} \delta^{i}_{k} - \rho_{G\,.k.j||m} \delta^{i}_{l} - \rho_{G\,.j.k.l||m} y^{i}.
\end{equation}
where \(B_j{}^i{}_{kl}\) is the Berwald curvature tensor of \(G\).

Using equations \eqref{Gi}, \eqref{Gij}, and \eqref{Gijk}, one derives
\[
B_j{}^{i}{}_{kl||m} = B_j{}^{i}{}_{kl|m} + B_j{}^{i}{}_{kl.r}(\rho_G \delta^{r}_{m} + \rho_{G\,.m} y^{r}) + B_r{}^{i}{}_{kl} (\rho_{G\,.m} \delta^{r}_{j} + \rho_{G\,.j} \delta^{r}_{m} + \rho_{G\,.j.m} y^{r})
\]
\[
+ B_j{}^{i}{}_{rl} (\rho_{G\,.m} \delta^{r}_{k} + \rho_{G\,.k} \delta^{r}_{m} + \rho_{G\,.k.m} y^{r}) + B_j{}^{i}{}_{kr} (\rho_{G\,.m} \delta^{r}_{l} + \rho_{G\,.l} \delta^{r}_{m} + \rho_{G\,.l.m} y^{r})
\]
\[
- B_j{}^{r}{}_{kl} (\rho_{G\,.r} \delta^{i}_{m} + \rho_{G\,.m} \delta^{i}_{r} + \rho_{G\,.r.m} y^{i}).
\]

Hence,
\begin{equation}\label{horizontalBGbar}
\begin{aligned}
B_j{}^{i}{}_{kl||m} = & \, B_j{}^{i}{}_{kl|m} + \rho_G B_j{}^{i}{}_{kl.m} + \rho_{G\,.m} B_j{}^{i}{}_{kl} + \rho_{G\,.j} B_m{}^{i}{}_{kl} \\
& + \rho_{G\,.k} B_j{}^{i}{}_{ml} + \rho_{G\,.l} B_j{}^{i}{}_{km} - \rho_{G\,.r} B_j{}^{r}{}_{kl} \delta^{i}_{m} - \rho_{G\,.m.r} B_j{}^{r}{}_{kl} y^{i}.
\end{aligned}
\end{equation}

Contracting with \(y^{m}\) yields
\begin{equation}\label{horizontalBGbar0}
B_j{}^{i}{}_{kl||m} y^{m} - B_j{}^{i}{}_{kl|m} y^{m} = -\rho_{G\,.r} B_j{}^{r}{}_{kl} y^{i}.
\end{equation}

Combining \eqref{horizontalBandBbarGbar} and \eqref{horizontalBGbar0}, we get
\begin{equation}\label{horizantalBGandGbar}
PB_j{}^{i}{}_{kl||0} = B_j{}^{i}{}_{kl|0} - \rho_{G\,.j.l||0} \delta^{i}_{k} - \rho_{G\,.k.l||0} \delta^{i}_{j} - \rho_{G\,.j.k||0} \delta^{i}_{l} - (\rho_{G\,.j.k.l||0} + \rho_{G\,.r} B_j{}^{r}{}_{kl}) y^{i}.
\end{equation}

Using \eqref{Rie-BerId}, we find
\[
PB_j{}^{i}{}_{kl||m} - PB_j{}^{i}{}_{km||l} = PR_j{}^{i}{}_{ml.k}.
\]

Contracting above with \(y^{m}\) and invoking \eqref{horizantalBGandGbar} leads to
\begin{equation}\label{GPRandR}
\begin{aligned}
PR_j{}^{i}{}_{ml.k} y^{m} = & \, B_j{}^{i}{}_{kl|0} - \rho_{G\,.j.l|0} \delta^{i}_{k} - \rho_{G\,.k.l|0} \delta^{i}_{j} - \rho_{G\,.j.k|0} \delta^{i}_{l} - \rho_{G\,.j.k.l|0} y^{i} \\
& - \Big[\rho_{G\,.r} B_j{}^{r}{}_{kl}-\big(\frac{\rho_G^{2}}{2}\big)_{.j.k.l} \Big] y^{i}.
\end{aligned}
\end{equation}
Here, we have used \eqref{pjklh0} that is proved in the following.

Note that \(\rho_{G\,.j.k||0} = \rho_{G\,.j.k|0}\). Using \eqref{Gij}, \eqref{Gijk}, and related calculations, we conclude
\begin{equation}\label{pjkhl}
\begin{aligned}
\rho_{G\,.j.k||l} = & \, \rho_{G\,.j.k|l} + \rho_G \rho_{G\,.j.k.l} + \rho_{G\,.l} \rho_{G\,.j.k} + \rho_{G\,.j} \rho_{G\,.l.k} + \rho_{G\,.k} \rho_{G\,.j.l} \\
= & \, \rho_{G\,.j.k|l} + \big( \frac{\rho_G^2}{2} \big)_{.j.k.l}.
\end{aligned}
\end{equation}

Additionally,
\begin{equation}\label{pjkh0}
\rho_{G\,.j.k||0} = \rho_{G\,.j.k|0},
\end{equation}
and
\begin{equation}\label{pjklh0}
\begin{aligned}
\rho_{G\,.j.k.l||0} = & \, \rho_{G\,.j.k||0.l} - \rho_{G\,.j.k||l} = \rho_{G\,.j.k|0.l} - \rho_{G\,.j.k|l} - \big( \frac{\rho_G^2}{2} \big)_{.j.k.l} \\
= & \, \rho_{G\,.j.k.l|0} - \big( \frac{\rho_G^2}{2} \big)_{.j.k.l}.
\end{aligned}
\end{equation}

From \eqref{Bjikl}, it follows that
\begin{equation}\label{GPRandEq}
PR_j{}^{i}{}_{ml.k} y^{m} = PB_j{}^{i}{}_{kl|0} - \rho_{G\,.r} (PB_j{}^{r}{}_{kl}) y^{i}.
\end{equation}

Then one finds that $F$ is generalized projective $R$-quadratic with respect to $\rho_G$, if and only if \eqref{GPR-quadEq}.
\end{proof}

This Theorem, proven in the above, establishes the necessary and sufficient condition for a Finsler metric to be $GPR$-quadratic characteristics. The expression for the projective Riemann curvature, as stated in \cite{Sh2}, is derived from a special case of a projective change of Finsler metric $F$, where $P=-\frac{S}{n+1}$, and can be written as $\bar{G}^i=G^i-\frac{S}{n+1}$. Then
\[
\bar{R}{^i}_k=R{^i}_k+E\delta{^i}_k+\tau_k y^i,
\]
where
\[
E=\frac{S^2}{(n+1)^2}+\frac{S_{|0}}{n+1}, \qquad \tau_k=3(\frac{S_{|k}}{n+1}+\frac{S S_{.k}}{(n+1)^2})+E_{.k}.
\]
Here $"|"$ denotes the horizontal derivative with respect to $F$.

Then according to the previous Theorem one can prove the Corollary \ref{PR-quadCor}.

\subsection{Proof of Corollary \ref{PR-quadCor}}

\begin{proof}
The straightforward conclusion arises from the aforementioned theorem. Taking into account that $E_{jk}=\frac{1}{2}S_{.j.k}$ and deriving from \ref{Bjikl} and $P=-\frac{S}{n+1}$, it becomes evident that
\[
\bar{B}_j{^i}_{kl}=B_j{^i}_{kl}-\frac{2}{n+1}\{E_{jk}\delta{^i}_l+E_{jl}\delta{^i}_k+E_{lk}\delta{^i}_j+E_{jkl}y{^i}\}=D_j{^i}_{kl}.
\]
\end{proof}

As in Corollary \ref{PR-quadCor}, it is of significant interest to determine the necessary and sufficient conditions under which the two generalized projective sprays introduced in Proposition \ref{rhoPropGS-cur} and Proposition \ref{rhoPropDV} are of $GPR$-quadratic type. The following theorems establish these conditions explicitly, providing a detailed characterization of when these sprays exhibit the $GPR$-quadratic structure.

\begin{thm}\label{GPRThmS-cur}
Let \(F\) be a Finsler metric on an \(n\)-dimensional manifold \(M\). Then \(F\) is of \(GPR\)-quadratic type with respect to \(\rho_G= \frac{\mathbb{S}}{n+1} + U\), where \(U = U(x,y)\) is any 1-homogeneous projective invariant function and \(\mathbb{S}\) is as introduced in Proposition \ref{rhoPropGS-cur}, if and only if the following holds,
\[
D_j{^i}_{kl \mid 0} - \frac{S_{.r}}{n+1} D_j{^r}_{kl} y^i = U_{.j.k \mid 0} \delta^i_l + U_{.j.l \mid 0} \delta^i_k + U_{.k.l \mid 0} \delta^i_j
\]
\[
+ \Big[ U_{.j.k \mid 0} \delta^i_l + U_{.j.l \mid 0} \delta^i_k + U_{.k.l \mid 0} \delta^i_j - \big(\tfrac{U^2}{2} + \frac{S}{n+1} U\big)_{.j.k.l} \Big] y^i.
\]
\end{thm}
\begin{proof}
Using Theorem \ref{Thm1} and noting that
\[
PB_{j{^i}_{kl}} = (PG^i)_{.j.k.l} = D_j{^i}_{kl} - U_{.j.k} \delta^i_l - U_{.j.l} \delta^i_k - U_{.k.l} \delta^i_j - U_{.j.k.l} y^i,
\]
we conclude that
\[
D_j{^i}_{kl \mid 0} - \frac{S_{.r}}{n+1} D_j{^r}_{kl} y^i = U_{.j.k \mid 0} \delta^i_l + U_{.j.l \mid 0} \delta^i_k + U_{.k.l \mid 0} \delta^i_j
\]
\[
+ U_{.j.k \mid 0} \delta^i_l + U_{.j.l \mid 0} \delta^i_k + U_{.k.l \mid 0} \delta^i_j- U_{.r} \big( U_{.j.k} \delta^r_l + U_{.j.l} \delta^r_k + U_{.k.l} \delta^r_j + U_{.j.k.l} y^r \big)
\]
\[
- \frac{1}{n+1} \big( U S_{.j.k.l} + U_{.j} S_{.k.l} + U_{.k} S_{.j.l} + U_{.l} S_{.j.k} + S U_{.j.k.l} + S_{.j} U_{.k.l} + S_{.k} U_{.j.l} + S_{.l} U_{.j.k} \big),
\]
which completes the proof of the theorem.
\end{proof}

The next theorem,
\begin{thm}\label{GPRThmDV}
Let \(F\) be a Finsler metric on an \(n\)-dimensional manifold \(M\). Then \(F\) is of \(GPR\)-quadratic type with respect to \(\rho_G= DV\), where \(V = V(x,y)\) is any 1-homogeneous directionally projective invariant function, if and only if the following holds,
\[
B_j{^i}_{kl \mid 0} - {\mathcal{D}V}_{.r} B_j{^r}_{kl}\ y^i = {\mathcal{D}V}_{.j.k \mid 0} \delta^i_l + {\mathcal{D}V}_{.j.l \mid 0} \delta^i_k + \big[{\mathcal{D}V}_{.k.l \mid 0} \delta^i_j
\]
\[
+ {\mathcal{D}V}_{.j.k.l\mid 0} + \frac{1}{2} ({\mathcal{D}V}^2)_{.j.k.l}\big] y^i.
\]
\end{thm}
\begin{proof}
The proof proceeds analogously to that of the previous theorem.
\end{proof}

In summary, this paper has successfully introduced a robust generalization of projective Riemann and Ricci curvatures within the framework of Finsler geometry. By defining and characterizing the class of generalized projectively $R$-quadratic and Ricci quadratic metrics, we have extended classical curvature notions to a broader geometric context, revealing new structural insights and invariance properties under projective transformations. These advancements lay the groundwork for further explorations into the geometric and analytic properties of Finsler spaces and open avenues for potential applications in theoretical physics and differential geometry. Future work may explore explicit constructions, examples, and applications of these generalized projective curvatures in diverse settings

\section*{Statements and Declarations}
The authors declare that no funds, grants, or other support were received during the preparation of this manuscript.
\end{document}